\documentclass[final,leqno]{siamltex1213}
\pdfoutput=1 
\usepackage{amsmath,blkarray,enumerate,
tabularx,url,multirow,xspace,array,xcolor}

\newcommand{\set}[1]{ \{\,#1\, \}}

\newcommand{\s}[1]{\text{\footnotesize$#1$}}

\newcommand{\clnl}{\multicolumn{1}{c}{}}

\newcommand{\offssymb}{\alpha}

\newcommand{\varoffs}[2]{{\offssymb_#1({#2})}}
\def\mod2pend{{\sc Mod2Pend}}

\newcommand{\setbg}[1]{\bigl\{\, #1\, \bigr\}}

\newcommand{\rng}[3]{#1 = {#2}\,{:}\,{#3}}

\newcommand\Tstr{\rule{0pt}{2.6ex}}       
\newcommand\Bstr{\rule[-1.2ex]{0pt}{0pt}} 
\newcommand{\slt}[1]{{\sl #1}\xspace}

\newcommand{\eqset}{I}
\newcommand{\varset}{J}

\newcommand{\TX} { \mathbf{x} }
\newcommand{\TF} { \mathbf{f} }

\newcommand{\ikb}[2]{\eqset_{#1}^{#2}}
\newcommand{\jkb}[2]{\varset_{#1}^{#2}}
\newcommand{\jk}[1]{\varset_{#1}}

\newcommand{\xjk}[1]{\TX_{\varset_{#1}}}
\newcommand{\xjkb}[2]{\TX_{\jkb{#1}{#2}}}
\newcommand{\xjkbig}[2]{\widetilde\TX_{\jkb{#1}{#2}}}

\newcommand{\fikb}[2]{\TF_{\ikb{#1}{#2}}}
\newcommand{\Fikb}[2]{\mathbf{F}_{\ikb{#1}{#2}}}

\newcommand{\eqntxt}[1]{\quad \text{#1} \quad}

\newcommand{\cf}{cf.\@\xspace}

\newcommand{\daesa}{{\sc daesa}\xspace}
\newcommand{\matlab}{{\sc Matlab}\xspace}
\newcommand{\daets}{{\sc daets}\xspace}

\newcommand{\JS}{\mathbf{J}} 

\newcommand{\defterm}[1]{[\S#1]}
\newcommand{\bcode}[2]{\sigma_#2(#1)}

\newcommand{\typeI}{{\tt I}\xspace}
\newcommand{\typeL}{{\tt L}\xspace}
\newcommand{\typeN}{{\tt N}\xspace}
\newcommand{\ttype}{{\tt Q}}
\newcommand{\vartype}{{\tt Q}}

\newcommand{\Jmat}[1]{A}
\newcommand{\bvec}[1]{b}

\newcommand{\yiset}{Y_i}

\newcommand{\ziset}{Z_i}
\newcommand{\kset}[1]{M_{#1}}
\newcommand{\Mset}[1]{M_}
\newcommand{\yset}[1]{Y_{#1}}

\newcommand{\assa}{\leftarrow}

\newcommand{\uv}{u}

\newcounter{vvc}
\newcommand{\vv}{{\rv_{\thevvc}}\refstepcounter{vvc}
}

\newcommand{\tagL}{0 }
\newcommand{\tagNL}{-1}
\newcommand{\tagIM}[1]{\inftag}
\newcommand{\tagMissing}[1]{\infty}
\newcommand{\inftag}{\infty}

\newcommand{\ivsindx}{{ {J_\text{v }}}}
\newcommand{\igsindx}{J_\text{g }}

\newcommand{\qlsym}{\gamma}
\newcommand{\qlvec}[1]{\qlsym_{#1}}
\newcommand{\wv}{v}
\newcommand{\rv}{v}

\newcommand{\nlbox}[1]{\colorbox{lightgray}{#1}}

\newcommand{\Input}{\text{}\\[-8pt]{\sc Input}}
\newcommand{\Output}{\text{}\\[-8pt]{\sc Output}}
\newcommand{\Compute}{\text{}\\[-8pt]{\sc Compute}}

\newcommand{\BAmcc}[3]{\BAmulticolumn{#1}{#2}{#3}}

\makeatletter
\def\rf#1{(\@rf#1,.)}
\def\@rf#1,{\ref{eq:#1}\@ifnextchar . {\@endrf}{, \@rf}}
\def\@endrf.{}
\makeatother

\newcommand{\rfr}[2]{(\ref{#1}--\ref{#2})}

\newcommand{\hk}{{k}_l}

\newcommand{\hc}{\widehat{c}}
\newcommand{\hd}{\widehat{d}}

\newcommand{\Setbg}[1]{ \bigl\{ #1   \bigr\}}

\newcommand{\blkl}[2]{#1\in B_{#2}}

\newcommand{\sij}[2]{\sigma_{#1#2}}
\newcommand{\sijc}[2]{\sigma_{#1,#2}}

\newcommand{\posinf}{\infty}
\newcommand{\neginf}{-\infty}
\newcommand{\lam}{\lambda}

\newcommand{\Kl}{K_l}

\newcommand{\sqr}[1]{#1^2}
\newcommand{\Diff}[2]{#1''}

\newcommand{\xvar}{\rv_{-5}}
\newcommand{\yvar}{\rv_{-4}}
\newcommand{\lamvar}{\rv_{-3}}
\newcommand{\uvar}{\rv_{-2}}
\newcommand{\wvar}{\rv_{-1}}
\newcommand{\muvar}{\rv_{0}}

\newenvironment{algo}[1]%
{
\begin{tabbing}{{\bf Algorithm.} \sc{#1}}\\%
\hspace*{2em}\=\hspace*{2em}\=\hspace*{2em}\=\hspace*{2em}\=\hspace*{2em}\=\hspace*{2em}\=\hspace*{2em}\=\hspace*{2em}\=\kill\!\!}
{\end{tabbing}}
\newcommand{\FOR}{{\bf for}\xspace}
\newcommand{\IF}{{\bf if}\xspace}
\newcommand{\ELSE}{{\bf else}\xspace}
\newcommand{\ELSEIF}{{\bf elseif}\xspace}
\newcommand{\THEN}{{\bf then}\xspace}
\newcommand{\CONTINUE}{{\bf continue}\xspace}
\newcommand{\TRUE}{{\bf true }\xspace}
\newcommand{\FALSE}{{\bf false }\xspace}

\newtheorem{example}{{\em Example}}[section]
\newtheorem{remark}{{\em Remark}}[section]

\definecolor{green1}{rgb}{0.06, 0.89, 0.94}
\definecolor{darkgray}{gray}{0.65}
\definecolor{pale-gray}{gray}{0.85}
\definecolor{LightCyan}{rgb}{0.88,1,1}
\definecolor{dgreen}{RGB}{1,120,1}
\definecolor{purple}{RGB}{230,0,95}
\definecolor{lgray}{rgb}{0.6,0.6,0.6}

\title{Exploiting Fine Block Triangularization and Quasilinearity in Differential-Algebraic Equation Systems}

\author{NEDIALKO S. NEDIALKOV\footnotemark[2]\ \footnotemark[4]
\and GUANGNING TAN\footnotemark[2]\ \footnotemark[5]
\and John D. Pryce\footnotemark[3]\footnotemark[6]
}

\begin{document}
\maketitle

\renewcommand{\thefootnote}{\fnsymbol{footnote}}

\footnotetext[2]{Department of Computing and Software, McMaster University, Hamilton, Canada}

\footnotetext[3]{Cardiff School of Mathematics, Cardiff University, UK}

\footnotetext[4]{Supported in part by the 
        Natural Sciences and Engineering Research Council of Canada (NSERC)}

\footnotetext[5]{Supported  in part by the  Ontario Research Fund (ORF), Canada}        

\footnotetext[6]{Supported  in part by The Leverhulme Trust}

\renewcommand{\thefootnote}{\arabic{footnote}}

\begin{abstract}
The $\Sigma$-method for structural analysis of a differential-algebraic equation  (DAE) system  produces offset vectors from which the sparsity pattern of DAE's system Jacobian is derived; this  pattern implies a fine block-triangular form (BTF).
This article derives a simple method for quasilinearity analysis of a DAE and  combines it with its fine BTF to construct a method for finding the minimal set of initial values needed for consistent initialization and a method 
for a block-wise computation of derivatives for the solution to the DAE.  
\end{abstract}



\pagestyle{myheadings}
\thispagestyle{plain}
\markboth{N. S. NEDIALKOV, G. TAN, AND J.D. PRYCE}{Exploiting Fine Block Triangularization and Quasilinearity in DAE Systems}

\section{Introduction}\label{sc:intro}

The authors have developed the \matlab package \daesa, Differential-Algebraic Equations Structural Analyzer \cite{NedialkovPryce2012b}, aimed at analyzing the structure of a system of differential-algebraic equations (DAEs) of the general form
\begin{align}\label{eq:maineq}
  f_i(\, t,\, \text{the $x_j$ and derivatives of them}\,) = 0, \quad i=1,
\ldots,n,
\end{align}
where the $x_j(t),\ j=1,\ldots,n$, are state variables, and $t$ is the time variable.
The $f_i$ can be arbitrary
expressions built from the $x_j$ and $t$ using $+,-,\times, \div$, other analytic standard functions, and the
$d^p/dt^p$ operator. 

  \daesa  implements the  {$\Sigma$-method} 
for structural analysis \cite{Pryce2001a}.
Using operator overloading, this package extracts the  {signature matrix} of \rf{maineq}, and then by solving 
a linear assignment problem,   finds two  {offset vectors}, 
from which it constructs  {coarse} and  {fine block-triangular forms (BTFs)} of the DAE.  Using the fine BTF, \daesa performs  {\em quasilinearity} (QL)
analysis and then finds the {\em minimal set} of variables and derivatives of them that require initial values, and also constructs a {\em block-wise solution scheme}.

Some of the theory of  these BTFs is presented in 
\cite{NedialkovPryce2012a}, where several results were left to be proved as future work. The companion article \cite{Pryce2014a} proves them and presents new results on BTFs, and in particular related to  the fine BTF. 
Describing the method for QL analysis was also left for  future work in \cite{NedialkovPryce2012a}: we derive this method here. We also present  \daesa's 
 algorithm for finding the minimal set of variables and derivatives  that need to be initialized and the algorithm for producing a block-wise solution scheme.

Section~\ref{sc:solscheme}
illustrates how the computation of derivatives for the solution to \rf{maineq} was prescribed originally by the  $\Sigma$-method.   Section~\ref{sc:btf} derives a  method for computing them based on   a  fine BTF of the DAE.
A simple method for QL analysis is derived in Section~\ref{sc:qla}. 
The overall solution scheme for computing derivatives for the solution to \rf{maineq}, building on its fine BTF and QL information, is given in Section~\ref{sc:overall}.
Conclusions are in Section~\ref{sc:concl}.

For brevity, we  refer to the  companion article \cite{Pryce2014a} for definitions and concepts. A term that is explained in \cite{Pryce2014a} is typeset here in slanted font on first occurrence, and the   subsection where it appears in \cite{Pryce2014a} is  referenced   
as \defterm{X}.

\smallskip

We assume that \rf{maineq} is \slt{structurally well posed}; that is, its \slt{signature matrix} $\Sigma = (\sij{i}{j}$) contains a \slt{highest-value transversal (HVT)} with entries $>-\infty$ \defterm{2.1}.

\section{Basic solution scheme}\label{sc:solscheme}
Let $c$ and $d$ be \slt{valid offset  vectors} \defterm{2.1} for \rf{maineq}, and let
 $k_d = -\max_j d_j$. We can find derivatives for the solution to 
 \rf{maineq} in stages  $k=k_d, k_d+1, \ldots$, where at stage $k$ we \begin{align}
\eqntxt{solve} &\setbg{f_i^{(k+c_i)} = 0\mid k+c_i\ge 0}
\label{eq:eqns}\\
\eqntxt{for} &\setbg{x_j^{(k+d_j)} \mid k+d_j\ge 0} 
\label{eq:vars}
\end{align}
{using} values for $\setbg{x_j^{(r)}\mid 0 \le r < k+d_j}$,
  which are found at stages $<k$ \cite{Pryce2001a}.
By a ``derivative'' $x_j^{(r)}$ we shall mean $x_j$ and (appropriate) derivatives of it.

We say the DAE \rf{maineq} is quasilinear (QL), if it is linear in the highest-order derivatives occurring in it, and non-quasilinear (NQL) otherwise (see also \S\ref{sc:qla}). 
To start this stage-wise pro\-cess, we need to initialize
\begin{align}\label{eq:ivnql}
\setbg{x_j^{(r)}\mid 0\le r \le d_j-\qlsym },
\eqntxt{where $\qlsym=1$ if the DAE is QL and 0 otherwise.}
\end{align}
We refer to \rf{eqns,vars} as  {basic (solution) scheme}. It succeeds (locally), if the \slt{System Jacobian} $\~J$, defined as
 $\~J_{ij} = \partial f_i/\partial x^{(\sij{i}{j})}$, if $\sij{i}{j} = d_j-c_i$ and $0$ otherwise,   is non-singular at a {consistent point}  
 \defterm{2.1}, see also \cite{Pryce2001a}. The systems at stages $k<0$ are generally underdetermined. For stages $k\ge 0$ they are square, and for $k > 0$ always linear, where the matrix of the linear system is  $\~J$. If the DAE is QL, the system at $k=0$ is also linear with a matrix $\~J$.

\smallskip 
 In practice, when solving \rf{maineq} numerically by  Taylor series, we  
compute 
 Taylor coefficients (TCs) $x_j^{(k+d_j)}/(k+d_j)!$ directly, where instead of derivatives in \rf{eqns,vars} we have TCs. Such a computation is implemented in the \daets solver;  see  \cite{nedialkov2005solving,nedialkov2007solving,nedialkov2008solving} for details. In the present work,  for simplicity of the exposition, we express the theory in terms of derivatives.

\begin{example}\label{ex:solscheme}\rm
Throughout this article, we use as an example 
the following DAE of differentiation  {index}  7:
\begin{align}
\begin{split}
0 = A &= x'' + x\lambda      \\
0 = B &= y'' + y\lambda + (x')^3 -G \\
0 = C &= x^{2} + y^{2} - L^{2} \\[1ex] 
\end{split}
\begin{split}
0 = D & = u'' + u\mu     \\
0 = E &  = (\wv''')^{2}  + \wv\mu   -G   \\
0 = F &= u^{2} + \wv^{2} - (L+c\lambda)^{2}+\lambda''.
\end{split}
\label{eq:mod2p}
\end{align}
The state variables are $x$, $y$, $\lam$, $u$, $\wv$, and $\mu$;  $L$ (length), $G$ (gravity), and $c>0$ are constants. These equations are obtained from a two-pendula problem \cite{NedialkovPryce2012b} in which 
\begin{align*}
\begin{split}
 B &= y'' + y\lambda -G, \quad
  E  = \wv''  + \wv\mu  - G,\eqntxt{and}
 F = u^{2} + \wv^{2} - (L+c\lambda)^{2}.
\end{split}
\end{align*}
%

The {$\Sigma$ matrix} of \rf{mod2p} and its $\~J$ are  shown in Figure~\ref{fig:sj}.
\begin{figure}[ht]
\begin{align*}
\Sigma &=
\begin{blockarray}{r@{\hskip 6pt}rc@{\hskip 6pt}c@{\hskip 6pt}c@{\hskip 6pt}c@{\hskip 6pt}c@{\hskip 6pt}c@{\hskip 6pt}c@{\hskip 6pt}c}
       & &  x_1 &  x_2 &  x_3 &  x_4 &  x_5 &  x_6 &  
       \\[1ex]
       & &  x &   y &  \lam &  u &  \wv &  \mu & \s{c_i} \\
\begin{block}{r@{\hskip 6pt}r[@{\hskip 3pt}c@{\hskip 6pt}c@{\hskip 6pt}c@{\hskip 6pt}c@{\hskip 6pt}c@{\hskip 6pt}c@{\hskip -3pt}]@{\hskip 6pt}c@{\hskip 6pt}c} \\[-2ex]
f_1&  A\;\;\;\; & 2^\bullet  &  & 0  &  &   &    &   \;\;\s4\;\;\\
f_2&  B\;\;\;\; & 1  &2   & 0^\bullet   &  &  &    &   \;\;\s4\;\; \\ 
f_3&  C\;\;\;\; & 0 &0^\bullet  &    &  &  &    &   \;\;\s6\;\;\\
f_4&  D\;\;\;\; &    &   &    &2  &  &0^\bullet   &   \;\;\s0\;\;\\
f_5&  E\;\;\;\; &    &   &      &  &3^\bullet    & 0 & \;\;\s0\;\;\\
f_6&  F\;\;\;\; &    &   &   2 &0^\bullet  &0 &    &  \;\;\s2\;\;  \\[1ex]
\end{block}
 &\s{d_j}& \s6 &\s6&\s4  &\s2&\s3&\s0
  \end{blockarray},
 &
 \~J &= \renewcommand{\arraystretch}{1.2}
\begin{blockarray}{c@{\hskip 6pt}c@{\hskip 6pt}c@{\hskip 6pt}c@{\hskip 6pt}c@{\hskip 6pt}c@{\hskip 6pt}c}
  &  x &y &\lam & u & v& \mu    \\
\begin{block}{c[@{\hskip 3pt}c@{\hskip 6pt}c@{\hskip 9pt}c@{\hskip 9pt}c@{\hskip 9pt}c@{\hskip 6pt}c@{\hspace{3pt}}]} \\[-2ex]
A\;\;\;\; &1 &   & x &  &   &  
\\
B\;\;\;\; &  & 1 & y &    &   &  
\\
C\;\;\;\; &2x & 2y &  &   &   &  
\\
D\;\;\;\; &  &   &  &  1 &   &  u
\\ 
E\;\;\;\; &  &   &  &   & 2\wv'''  &  \wv
\\
F\;\;\;\; &  &   &  1 &  2u &   &  
\\
\end{block} 
\end{blockarray}\\[-1ex]
&\hspace{63pt}\text{blank denotes $\neginf$}
&&\hspace{55pt}\text{blank denotes 0}
\end{align*}
\vspace{-20pt}
\caption{\label{fig:sj}Signature matrix and system Jacobian of  \protect\rf{mod2p} with labeling of equations, variables and offsets. A HVT in $\Sigma$ is marked with $\bullet$.
$\JS_{2,1}=0$ and $\JS_{6,5} =0$, since $2 = d_1-c_2> \sigma_{2,1}=1$ and $1 = d_5-c_6> \sigma_{6,5}=0$, respectively.}
\end{figure}
In Table~\ref{tbl:solscheme}, we illustrate the basic scheme  \rf{eqns,vars}   when applied to  \rf{mod2p}.
This problem   is NQL (because of $(v''')^2$), so $\qlsym=0$, and \rf{ivnql} implies we need to give initial values for $x^{(\le 6)}$, $y^{(\le 6)}$, $\lam^{(\le 4)}$, $u^{(\le 2)}$, $\wv^{(\le 3)}$, and $\mu$;
the notation $z^{(\le r)}$ is short  for $z, z', \ldots, z^{(r)}$. We show in \S\ref{ss:init} how the number of initial values can be drastically reduced by exploiting a 
{fine BTF} of \rf{mod2p}.

\begin{table}[ht]\footnotesize
\caption{Basic scheme for computing derivatives of the solution to (\ref{eq:mod2p}). Nonlinear equations and the variables that appear nonlinearly in them are marked by
$\nlbox{\protect\phantom{u}}$\,.}
\vspace{-10pt}
\label{tbl:solscheme}
\begin{equation*}
\renewcommand{\arraystretch}{1.1}
\begin{array}
{lc | l |l|l|l|l|l|l|cl }
\multicolumn{11}{c}{\text{stage $k$}}\\
\cline{3-11}
& c_i,\, d_j& \multicolumn{1}{c}{-6} & \multicolumn{1}{c}{-5} &   \multicolumn{1}{c}{-4} & \multicolumn{1}{c}{-3} &\multicolumn{1}{c}{-2} & \multicolumn{1}{c}{-1} &\multicolumn{1}{c}{0}&\cdots  &  \multicolumn{1}{c}{>0}
\\ \cline{2-11}
 \Tstr
 \multirow{6}{*}{solve}
		  &4& &  &A  & A'  &\text{$A''$}&A'''& A^{(4)}       &\cdots & A^{(k+4)}  \\ 
	&4 &   &  & B     & B'       & B''		&B'''& B^{(4)}       &\cdots &B^{(k+4)}   \\  
	&6	&  \nlbox{$C$}  & C'  & C''     & C'''      & C^{(4)} &C^{(5)} &C^{(6)}  &  \cdots &    C^{(k+6)}   
	 \\ 
&0	&    &    &&&& &D  & \cdots &D^{(k)}   \\ 
&0	&    &    &&&& &\nlbox{$E$}  & \cdots &E^{(k)}   \\ 
&2&    &    & &&\nlbox{$F$}     & F '      &F''  &\cdots &F^{(k+2)}    \Bstr 
\\ \cline{2-11}
\Tstr
\multirow{6}{*}{for}&
6&    \nlbox{$x$}   & x'  & x''     & x'''       &x^{(4)}  &x^{(5)}& x^{(6)}       &\cdots &x^{(k+6)}   
\\ 
&6  &   \nlbox{$y$}   & y'  & y''     & y'''       &y^{(4)}  &y^{(5)}& y^{(6)}       &\cdots &y^{(k+6)}    
  \\  
&4 &         & &\lam  & \lam'  &\lam''&\lam'''& \lam^{(4)}       &\cdots &\lam^{(k+4)}   
 \\
&2&    &     && &\nlbox{$u$} & u'  &u''&\cdots &u^{(k+2)}       \\
&3&         && &\wv & \wv'  &\wv''&\nlbox{$\wv'''$}& \cdots &\wv^{(k+3)}    \\
&0&         &&&&& &\mu & \cdots &\mu^{(k)}   \Bstr       
\\
 \cline{1-11}
\end{array} 
\end{equation*}
\vspace{-5pt}
\end{table}

\end{example}

\section{Solution scheme through fine BTF}
\label{sc:btf}

The \slt{coarse BTF}     of  \rf{maineq} is based on the sparsity pattern
of the DAE:
$
S = \setbg{(i,j)\mid \sij{i}{j}>\neginf}
$.
For given  valid offset vectors  $c$ and $d$, 
the corresponding  \slt{fine BTF} \defterm{4.1} is based on the sparsity pattern of~
$\~J$:
\begin{align}
{S_0=S_0(c,d)} = \setbg{(i,j)\mid d_j-c_i=\sij{i}{j}}.\label{eq:S0}
\end{align}
We refer to $c$ and $d$ as \slt{global offsets}.
 By $\hc$ and $\hd$ we denote the vectors of \slt{local offsets}. 
We assume that both global and local offsets are valid  but not necessarily \slt{canonical}  \defterm{5.1}, except in \S\ref{ss:init}, where the local offsets must  
 canonical. 
 
 \begin{example}\rm
For \rf{mod2p} there are two coarse (diagonal) blocks, see  Figure~\ref{fig:sjb}, consisting of equations $E,D,F=0$ in variables  $\wv, \mu, u$, 
and equations $A,B,C=0$ in variables $x,y,\lam$, respectively. 
The former can be decomposed into three fine blocks, while the latter cannot  be  decomposed into smaller fine blocks, as the block form of its sparsity pattern, $S_0$, is irreducible.

\begin{figure}[ht]
\newcommand{\dbar}{\BAmcc{1}{|c}{}}
\newcommand{\lbar}[1]{\BAmcc{1}{|c}{#1}}

\begin{align*}\renewcommand{\arraystretch}{1.2}
\Sigma=\begin{blockarray}{r@{\hskip 5pt}rc@{\hskip 5pt}c@{\hskip 5pt}c@{\hskip 5pt}c@{\hskip 5pt}c@{\hskip 5pt}c@{\hskip 5pt}c@{\hskip 5pt}c}
&&  x_1 &   x_2 &  x_3 &  x_4 &  x_5 &  x_6\\[1ex]
       &&  \wv &   \mu &  u &  x &  y &  \lam & \s{c_i} & \s{\hc_i}\\
\begin{block}{r@{\hskip 5pt}r[@{\hskip 3pt}c@{\hskip 5pt}c@{\hskip 5pt}c@{\hskip 5pt}c@{\hskip 5pt}c@{\hskip 5pt}c@{\hskip -3pt}]@{\hskip 6pt}c@{\hskip 5pt}c} \\[-2ex]
f_1 &E\;\;\;\; & 3  & \lbar{0}  &  & \dbar &   &    &   \;\;\s0\;\; &\s0 \\\cline{3-4}
f_2 &D\;\;\;\; &    & \lbar{0}   & \lbar{2} &  \dbar &  &    &   \;\;\s0\;\; &\s0\\ \cline{4-5}
f_3 &F\;\;\;\; & 0 &  &\BAmcc{1}{|c}0    & \dbar &  &   2 & \;\;\s2\;\; &\s0\\ \cline{3-8}
f_4  &A\;\;\;\; &    &   &    &\BAmcc{1}{|c}2 & &0    & \;\;\s4\;\;  &\s0\\
f_5  &B\;\;\;\; &    &   &      & \BAmcc{1}{|c}1 &\; 2   & 0 & \;\;\s4\;\; &\s0\\
f_6  &C\;\;\;\; &    &   &    &\BAmcc{1}{|c}0 & \; 0 &   &   \;\;\s6\;\; &\s2\\
[2pt]
\end{block}
&\s{d_j}& \s3 &\s0&\s2 &\s6&\s6&\s4\\
&\s{\hd_j}& \s3 &\s0&\s0 &\s2&\s2&\s0
\end{blockarray},
\quad
\renewcommand{\arraystretch}{1.4}
\~J=
\begin{blockarray}{c@{\hskip 5pt}c@{\hskip 5pt}c@{\hskip 5pt}c@{\hskip 5pt}c@{\hskip 5pt}c@{\hskip 5pt}c}
 &  \wv &   \mu &  u &  x &  y &  \;\lam \\
\begin{block}{c[@{\hskip 3pt}c@{\hskip 5pt}c@{\hskip 5pt}cc@{\hskip 5pt}c@{\hskip 5pt}c@{\hspace{3pt}}]} \\[-2ex]
E\;\;\;\; & 2v''' & \lbar{\wv}   &  & \dbar &   & \\\cline{2-3}
D\;\;\;\; &  & \lbar{u} & \lbar{1} & \dbar &   &
\\   \cline{3-4}
F\;\;\;\; &  &   &\lbar{2u}   & \dbar  & &  1
\\ \cline{2-7}
A\;\;\;\; &&&& \lbar{1} &   & x
\\
B\;\;\;\; &&&& \dbar  & 1 & y
\\
C\;\;\;\; &&&& \lbar{2x} & 2y & 
\\
\end{block}
\end{blockarray}
\end{align*}
\vspace{-20pt}
\caption{\label{fig:sjb}Permuted $\Sigma$ and $\~J$ of \protect\rf{mod2p} into fine BTF;
$c_i$ and $d_j$ are  {global canonical offsets},
and  $\hc_i$ and $\hd_j$ are {local canonical offsets.}}
\end{figure}
\end{example}

\renewcommand{\^}[1]{^{(#1)}}

In the solution scheme below, we exploit the fine BTF 
(corresponding to the given global offsets):
we assume that the equations and variables of 
\rf{maineq} are permuted such that the resulting $\Sigma$ and $\~J$ are in fine BTF. For every position $(i,j)$ below a diagonal block, $d_j-c_i>\sij{i}{j}$, and hence $\~J_{i,j}$ is identically zero.

 For simplicity in the notation, we denote the permuted equations and variables as before, 
$f_1,f_2,\ldots, f_n$ and $x_1,x_2,\ldots, x_n$. (In the companion paper \cite{Pryce2014a}   they are denoted by $\widetilde f_i$ and $\widetilde x_i$, $\rng{i}{1}{n}$.)

\subsection{Solution scheme by blocks}\label{ss:ssblocks}

Assume that there are $p$ fine diagonal blocks, each of size $N_l$. Denote by $B_l$ the set of indices of rows [resp. columns] in block $l$. That is, 
\[
\textstyle{B_l  = \left\{ i  \mid   \sum_{r=1}^{l-1}N_r < i \le 
\sum_{r=1}^{l}N_r\right\} .}
\]
For example in Figure~\ref{fig:sjb}, $\blkl{2}{2}$
and  $\blkl{5}{4}$. Denote also
$$
B_{l:q} = \bigcup_{r=l}^{q} B_r .
$$

We 
re-organize the basic scheme \rfr{eq:eqns}{eq:vars} as follows.
At stage $k$,  
we solve blocks $l = p, p-1, \ldots, 1$ in order such that for block $l$ we 
\begin{align}
\eqntxt{solve} &\ \Setbg{f_i^{(k+c_i)} = 0\mid \blkl{i}{l} \text{ and }k+c_i\ge 0} 
\label{eq:eqnsbl}
\\
\eqntxt{for} &\ \Setbg{x_j^{(k+d_j)} \mid 
\blkl{j}{l} \text{ and } k+d_j\ge 0}
\label{eq:varsbl}
\\
\text{using}&&& \nonumber\\
\renewcommand{\arraystretch}{1.4}
&\begin{array}{lr}\Setbg{x_j^{(r)}\mid 0 \le r < k+d_j} & \qquad\text{(computed at stages $<k$) and}\\[1ex]
\multicolumn2l{\Setbg{x_j^{(k+d_j)} \mid 
\blkl{j}{l+1:p} \text{ and } k+d_j\ge 0} }\\
\multicolumn2r{\text{(available from blocks $>l$ at this stage).}}
\end{array}\label{eq:usingbl}
\end{align}
We refer to (\ref{eq:eqnsbl}--\ref{eq:usingbl}) as {block (solution) scheme}.

\begin{example}\rm
Using the fine BTF from Figure~\ref{fig:sjb}, we illustrate in Table~\ref{tbl:solschbtf} the scheme  (\ref{eq:eqnsbl}--\ref{eq:usingbl}) when applied to \rf{mod2p}.
\begin{table}[ht]
\footnotesize
\caption{Block scheme for computing derivatives for the solution to \protect\rf{mod2p}.\label{tbl:solschbtf}}
\vspace{-10pt}
\centering
\begin{equation*}
\renewcommand{\arraystretch}{1.4}
\begin{array}{cl|c|c|c|c|c|c|c|c|cc}
 \multicolumn{1}{c}{} & \multicolumn{1}{c}{}& \multicolumn{1}{c}{}&\multicolumn{7}{c}{\text{stage $k$}}\\ 
 \cline{4-12}
\multirow{8}{*}{block 4} 
&\multicolumn{1}{c}{}& \multicolumn{1}{c|}{c_i,d_j} &\multicolumn{1}{c}{-6} & \multicolumn{1}{c}{-5} &   \multicolumn{1}{c}{-4} &\multicolumn{1}{c}{-3} & \multicolumn{1}{c}{-2} & \multicolumn{1}{c}{-1} &
\multicolumn{1}{c}{0}&
\cdots &
\multicolumn{1}{c}{>0}\\ \cline{1-12}
\Tstr
	&&  4 &&     &A  & A'  &A''&A'''& A^{(4)}       &\cdots &A^{(k+4)}  
	\\
	& \text{solve}&  4 & & & B     & B'       &B''  		&B'''& B^{(4)}       &\cdots &B^{(k+4)}  \\  
	&&6& \nlbox{$C$}  & C'  & C''     & C'''      &C^{(4)}  &C^{(5)}& C^{(6)}       &\cdots &C^{(k+6)} 
	\Bstr   \\ 
	\cline{3-12}\Tstr
&   & 6& \nlbox{$x$}  & x'  & x''     & x'''       &x^{(4)}  &x^{(5)}& x^{(6)}       &\cdots &x^{(k+6)}    
\\ 
  &\text{for}&  6&  \nlbox{$y$}   & y'  & y''     & y'''       &y^{(4)}  &y^{(5)}& y^{(6)}       &\cdots &y^{(k+6)}   \\  
&    &     4& &&\lam  & \lam'  &\lam''&\lam'''& \lam^{(4)}       &\cdots &\lam^{(k+4)}  \\
 \hline\hline
 \multirow{4}{*}{block 3}
   &    &     &\clnl& \clnl  &\clnl&   &\downarrow&    \downarrow& \downarrow       &\cdots &\downarrow      \\
&\text{solve}&2&\clnl& \clnl   	   &	\clnl       &                  &  \nlbox{$F$}	  & {F'}& F''      &  \cdots &F^{(k)}   \\ 
\cline{8-12}\Tstr
&\text{for}&  2&\clnl &\clnl& \clnl	   	       &                  &  \nlbox{$u$} 	  & {u'}& u''       &\cdots & u^{(k)}  \\ 
 \hline\hline
 \multirow{4}{*}{block 2}
 & & &\clnl &\clnl&\clnl& \clnl& \clnl  &&\downarrow& \cdots&\downarrow  \\
%
&\text{solve}& 0 &\clnl&  \clnl     	   	       &   \clnl             & \clnl & \clnl 	  & &D   &      \cdots&D^{(k)}    \\  \cline{10-12}\Tstr
&\text{for}&0&\clnl  & \clnl      	   	       &   \clnl               & \clnl 	& \clnl & &\mu   &     \cdots&\mu^{(k)}   \\
\hline\hline 
\multirow{4}{*}{block 1}
&&&\clnl &\clnl &\clnl& \clnl&\clnl  &&\downarrow &\cdots&\downarrow  \\
%
&\text{solve} &0&   \clnl               &   \clnl               & \clnl&\clnl	  &\clnl  & &\nlbox{$E$}  &    \cdots&E^{(k)}  \\ 
\cline{7-12}\Tstr
&\text{for} &3& \clnl&\clnl                                   &&	\wv  & \wv' & \wv''& \nlbox{$\wv'''$}    &    \cdots&\wv^{(k+3)}     \\ \hline
\end{array} 
\end{equation*}

\end{table}
As we discuss in \S\ref{ss:init}, for this problem 
 we need to give initial guesses for 
$x$, $x'$, $y$, $y'$, $u$, $\wv'''$, and initial values for $\wv, \wv', \wv''$.

Using block 4, we can compute  derivatives for $x,y,\lam$ independently of the other blocks. At stage $-2$, as soon 
as $\lam''$ is available, we can find $u$ in block 3 by solving  (nonlinear) $F=0$. At stage 0, as soon as $u''$ is available from block 3, we can find $\mu$ in block 2  by solving  (linear) $D=0$, and then find $\wv'''$ in block 1 by solving  (nonlinear) $E=0$. When considering block 1, at each stage $-3$, $-2$, and $-1$   there is no equation to solve  for $\wv$, $\wv'$, and $\wv''$, respectively, and initial values are required for them. 
 \end{example}

\subsection{Block scheme more formally}

Denote
\begin{align}\label{eq:ikl}
\ikb{k}{l} &= \setbg{ (i,r)  \mid  \blkl{i}{l} \text{ and }r =  k+c_i\ge0},
\\
\jkb{k}{l} &= \setbg{ (j,r) \mid  \blkl{j}{l}
\text{ and } r=k+d_j\ge0}, \label{eq:jkl}
\\
\jkb{<k}{}&=\Setbg{(j,r)\mid 0 \le r < k+d_j},
\eqntxt{and}\label{eq:jkl_lk}
\\
\begin{split}
 \jkb{k}{>l}&=\Setbg{(j,r) \mid 
\blkl{j}{l+1:p} \text{ and } r=k+d_j\ge 0}\\
&= \bigcup _{r>l}\jkb{k}{r}. 
\end{split} \label{eq:jk_gtl}
\end{align}
By $\xjkb{k}{l}$ we mean a vector with components 
 the elements of the set\footnote{Some ordering must be agreed on but it does not matter.} 
$$\setbg{x_j^{(r)} \mid (j,r)\in \jkb{k}{l}};$$  similarly for 
$\xjkb{<k}{\phantom{l}}$ and $ \xjkb{k}{>l}$.
By $\fikb{k}{l}$ we denote a vector with 
components $$\setbg{f_i^{(r)} \mid (i,r)\in \ikb{k}{l}}.$$

Then (\ref{eq:eqnsbl}--\ref{eq:usingbl}) can be written concisely as: at stage $k$ solve in order for blocks $l=p, p-1, \ldots, 1$ the system
\begin{align}
\fikb{k}{l}(t, \xjkb{<k}{\phantom{l}}, \xjkb{k}{l}, \xjkb{k}{>l})=0
\label{eq:solb}
\end{align}
{for} $\xjkb{k}{l}$, where   $\xjkb{<k}{\phantom{l}}$ is found during previous stages,
and $\xjkb{k}{>l}$ is found at this stage but from  previous blocks.
\begin{example}\rm
\rm To illustrate the above, consider stage $k=-2$. The sets \rf{ikl,jkl,jk_gtl} are ({blank denotes $\emptyset$})
\begin{align*}\renewcommand{\arraystretch}{1.2}
\renewcommand{\setbg}[1]{\{#1\}}
\begin{array}{c|c|c|c|}
\multicolumn{1}{c}l &\multicolumn{1}{c}{\ikb{-2}{l}}&
\multicolumn{1}{c}{\jkb{-2}{l}}&\multicolumn{1}{c}{\jkb{-2}{>l}}\Bstr
\\ \hline
4& \setbg{(4,2), (5,2), (6,4)} &\setbg{(4,4), (5,4), (6,2)} &\\
3& \setbg{{(3,0)} }& \setbg{(3,0)} & \setbg{(6,2)}\\
2& &&\\
1&  & \setbg{(1,1)}& 
\end{array}
\end{align*}
and \rf{jkl_lk} is 
$\jk{<-2} = \setbg{(1,0), (4,<4)(4,<4), (5,<4), (6,<2)}$,
where $(j,<r)$ is short  for $(j, 0), (j,1), \ldots (j,r-1)$. 
Then 
\rf{solb} can be illustrated as
\begin{align*}\renewcommand{\arraystretch}{1.2}
\begin{array}{c|c|c|c|c|c}
\multicolumn{1}{c}l & \multicolumn{1}{c}{\fikb{-2}{l}}& 
\multicolumn{1}{c}{\xjkb{-2}{l}}&\multicolumn{1}{c}{\xjkb{-2}{>l}}
\\[1ex]  \hline
\Tstr
4& A'', B'', C^{(4)}& 
 { x^{(4)}, y^{(4)}, \lam''}&
 \\
3&F&  {u} &  { \lam''}\\
2&  &&\\
1&  &\wv'&\\ 
\end{array}
\end{align*}
where $\xjk{<-2} = \bigl(\wv;\, x,x', x'', x'''; \,y,y',y'', y'''; \,\lam, \lam')$.
Note that in  block 1, there is nothing to solve, and we give an (arbitrary) initial value for $\wv'$.
\end{example}

\smallskip

For each fine block $l$, the difference between global and local offsets is a non-negative constant, the \slt{lead time} \defterm{5.1} of this block: 
$$d_i-\hd_i=c_i-\hc_i = \Kl\ge 0\eqntxt{for all $\blkl{i}{l}$.}$$
Let
$\hk= k + K_l$, which we refer to as local stage.
Then 
\begin{align}
k+c_i = \hk+\hc_i\eqntxt{and}
k+d_j = \hk+\hd_j.
\label{eq:kkl}
\end{align}
Using \rf{kkl}, we  write  \rf{eqnsbl,varsbl} as 
\begin{align}
\eqntxt{solve} &\Setbg{f_i^{(\hk+\hc_i)} = 0\mid \blkl{i}{l} \text{ and }\hk+\hc_i\ge 0}
\label{eq:eqnsbl2}\\
\eqntxt{for} &\Setbg{x_j^{(\hk+\hd_j)} \mid 
\blkl{j}{l} \text{ and } \hk+\hd_j\ge 0}.
\label{eq:varsbl2}
\end{align}
For $k$ such that $k_l = k+K_l\ge0$, from \rf{eqnsbl2,varsbl2}, we have a square system, 
and therefore \rf{solb} is square. For $k_l>0$ it is always linear in its highest-order derivatives, as each of the equations is differentiated at least once.

\begin{lemma}\label{lemma:underdet}
If $\min_{\blkl{i}{l}} \hc_i =0$, then system \rf{solb} is underdetermined for any {$k$ such that $k_l=k+\Kl<0$}.
\end{lemma}

 \begin{proof}
%
%
Since the sparsity pattern $S_0$ \rf{S0} of a fine block is irreducible, by the Strong Hall Property  
\cite[\S4.2]{Pryce2014a}, any  set of 
$r\le N_l-1$ columns [resp. rows] contains  
elements of at least $r+1$ rows [resp. columns]. Therefore, any set of $r \le N_l-1$ equations contains at least $r+1$ variables $x_j$ with $\sij{i}{j}=d_j-c_i$. The highest-order derivative of such a $x_j$ in $ f_i^{(k + c_i)} = f_i^{(\hk+\hc_i)}$
is 
\[
\sij{i}{j}+ \hk+\hc_i = \hd_j-\hc_i + k+\hc_i = \hk+\hd_j\ge 0.
\]
Since at least one $\hk+\hc_i  <  0$, where $\blkl{i}{l}$, at local stage   $k_l < 0$ there are $r\le N_l-1$ equations with at least $r+1$ variables.
\qquad\end{proof}

\begin{remark}\rm
The condition $\min_{\blkl{i}{l}} \hc_i =0$ ensures that the local offsets for block $l$ are \slt{normalized} \defterm{2.1}. Canonical offsets are normalized, but not vice versa; see \cite{Pryce2014a}.
\end{remark}
\smallskip

For brevity, write \rf{solb} as 
\begin{align}
\Fikb{k}{l}(\xjkb{k}{l}) = 0.\label{eq:solb2}
\end{align}
Let $k_l= k+\Kl <0$. If $\ikb{k}{l} = \emptyset$, then there are no equations at this stage, and we  simply give initial values for $\xjkb{k}{l}$; otherwise we need initial guesses (trial values for a nonlinear solver)  $\xjkbig{k}{l}$ to find $\xjkb{k}{l}$ by solving
\[
\min \| \xjkbig{k}{l} - \xjkb{k}{l}\|_2\eqntxt{s.t.} \Fikb{k}{l}(\xjkb{k}{l}) = 0.
\]

We derive in  \S\ref{ss:init} an algorithm for determining which derivatives to initialize and in \S\ref{ss:bsa} an algorithm for the overall solution scheme for \rf{solb2}.

\section{Quasilinearity analysis}\label{sc:qla}

  Consider the basic scheme.
When $k>0$, an $f_i^{(k+c_i)}$ with $k+c_i>0$ is linear in its highest-order derivatives.
If $k\le 0$ and $c_i$ is such that $k+c_i=0$, 
the corresponding $f_i$ can be nonlinear in the derivatives we solve for, which are the $x_j^{(\sij{i}{j})}$ with $\sij{i}{j}=d_j-c_i$.

For an $f_i$, let  
\begin{align}\label{eq:yiset}
\yiset &= \setbg{ x_j^{(\sij{i}{j})} \mid \sij{i}{j} =d_j-c_i}.
\end{align}
%

\begin{definition}\label{def:fiql}
Equation $f_i=0$ is QL, if it is  linear in all 
$y\in\yiset$, and NQL otherwise.
\end{definition} 

\begin{definition}\label{def:daeql}
System \rf{eqns} is NQL (at stage $k$), if it contains an undifferentiated NQL $f_i$, and QL otherwise.
The DAE \rf{maineq} is NQL, if it is NQL at stage 0, and QL otherwise.
\end{definition}

\begin{example}\rm Consider \rf{mod2p} and Table~\ref{tbl:solscheme}. At stage $k=-2$, we have $k+c_i=0$ only for equation  $f_6 = F=0$.
%
Here, $Y_6 = \set{\lam'', u}$ as $x_3^{(d_3-c_{6})} = x_3^{(4-2)}= \lam''$ and 
$x_4^{(d_4-c_{6})} = x_4^{(2-2)} = u$; $x_1^{(d_5-c_{6})} = x_5' = v'$  does not appear in $F$ (note $1 = d_5-c_{6} >\sijc{6}{5} = 0$). 
This equation and the system are NQL at stage $k=-2$ .

\end{example}

 Our goal is to determine automatically if a system at stage $k\le 0$ is QL. (For stages $k>0$, they are always QL.)
 We achieve this by propagating 
 the  {\em offset} and {\em QLity} (pronounced {\em cuellity}) for each   variable in 
 a code list of the DAE.
 In \S\ref{ss:codelist},  we specify what we mean by code list, and in \S\ref{ss:varoffs} and \S\ref{ss:vartype}, we introduce the offset and QLity of a variable.
 The derivations in these two subsections are summarized in the algorithm in \S\ref{ss:algor}.
  In \S\ref{ss:qlafine}, we give a simple modification to it, so it works  correctly for the block scheme. 
 Appendix \S\ref{ss:oov} suggests a simple approach for an implementation of QL analysis using operator overloading, and also suitable for source code translation.

\pagebreak
\subsection{Code list}\label{ss:codelist}

An $f_i$ is described by an expression containing arithmetic operations, standard functions, and the $d^p/dt^p$ operator. Such an expression can be represented by {code list}
containing input, intermediate, and an output variables as follows.
The input variables are $t$ and $x_j$ for $\rng{j}{1}{n}$.
We rename them as $v_{-n } = t$ and $\rv_{j-n} = x_j$.
Then each subsequent variable $\rv_r$ for $r>0$ is 
defined using previous variables and an operation 
of arity $m> 0$:
\begin{align}
\rv_r = \phi_r (\rv_{i_1}, \rv_{i_2}, \ldots, \rv_{i_m}), \quad 
-n \le i_q < r \eqntxt{for each $\rng{q}{1}{m}$.}
\label{eq:code}
\end{align}
$\phi_r $ can be an arithmetic operation, elementary function, the identity function, $d^p/dt^p$, or a user-defined function.
We refer to a constant (an operation of arity 0) directly instead of assigning it to a variable. 
We refer to the last variable in the code list of an $f_i$ as an output variable. 

Assuming that we evaluate all the $f_i$'s and that the last $n$ variables are output variables, we can illustrate the above  as 
\newcommand{\noivs}{q}
\[
\bigl [\, \underbrace{\rv_{-n}}_t, \underbrace{\rv_{j-n}, \ldots, \rv_{0}}_{x_j}, \ 
\
\rv_1, \ldots, \rv_{\noivs},\   
\underbrace{\rv_{\noivs+1}, \ldots, \rv_{\noivs+n}}_{f_i}\, \bigr];
\]
 $\rv_1, \ldots, \rv_q$  are intermediate variables.

\subsection{Offset of a variable}\label{ss:varoffs}

Let $\rv$ be a variable\footnote{This $v$ and later $u$ are not to be confused with the $v$ and $u$ in \rf{mod2p}.} in a code list of an $f_i$. 
\begin{definition}\label{def:sigvec}
The {\em signature vector}
of $\rv$  is the $n$ vector with $j$th component 
\[
\sigma_j(\rv) = \begin{cases}
\text{the highest-order of derivative of $x_j$ on which $\rv$ formally depends, or}\\
\neginf \text{\ \ if $\rv$ does not depend on $x_j$.}\\
\end{cases}
\]
\end{definition}
Note that $\sigma_j(\rv)\le \sij{i}{j}$. ``Formally'' means that we do not consider symbolic simplifications in the expressions that arise; e.g. the highest-order derivative  of $x$ in 
$x''' + x'' + \lam x - x'''$ is 3 not 2. (See  
\cite{nedialkov2007solving} for more details and an algorithm for the computation and propagation of these vectors to compute the signature matrix of a DAE.)

For an equation number $i$, denote by $\Mset_i$ the set of indices for which $\sij{i}{j} = d_j-c_i$:
\begin{align*}
\Mset_i   & = \setbg{  j \mid  \sij{i}{j}=d_j -c_i }.
\end{align*}
Since we assume the DAE is structurally well posed,  each entry in a HVT is $\ge 0$, and therefore
 $\sij{i}{j}=d_j -c_i$ for at least one $j$. Hence $M_i\neq\emptyset$.
\begin{definition}\label{def:offsetvec}
The {\em offset} of $\rv$, with respect to $f_i$, is
\begin{equation}\label{eq:offsetdef}
\varoffs{i}{\rv}= 
\min_{j\in \Mset_i}{\bigl(\sij{i}{j} - \bcode{\rv}{j}\bigr)}.
 \end{equation}
\end{definition}
%

\begin{example}\rm
For $f_6 = F$, $\Mset_6 = \set{3,4}$.
Then 
\rf{offsetdef} becomes
\begin{align*}
\varoffs{6}{\rv}&=
\min_{j=3,4}{\bigl(\sijc{6}{j} - \bcode{\rv}{j}\bigr)}
= \min\setbg{  \sijc{6}{3} - \bcode{\rv}{3} , \,
 \sijc{6}{4} - \bcode{\rv}{4} }\\
& = \min\setbg{  2 - \bcode{\rv}{3} , \,
  - \bcode{\rv}{4}}.
\end{align*}
Table~\ref{tbl:offsets} shows, for the code list in the first column, the corresponding $\sigma_3(v_r)$, $\sigma_4(v_r)$, and   $\varoffs{6}{v_r}$.

\begin{table}[ht]\footnotesize
\caption{Code list and offsets for $f_6 = F$
\label{tbl:offsets}}
\vspace{-10pt}
\begin{align*}
\begin{array}{lll@{\hskip 12pt}lrrr}
\multicolumn{3}{c}{\text{code list}} &\multicolumn{1}{c}{\text{expression}} &\multicolumn{1}{c}{\sigma_{3}(v_r)} & \multicolumn{1}{c}{\sigma_{4}(v_r)} & \multicolumn{1}{c}{\alpha_{6}(v_r)}\\ \hline
\Tstr
&\lamvar &= \lam  & \lam&0 & \neginf &2\\
& \uvar &= u&u  &\neginf & 0 & 0\\
&\wvar &= v& v&\neginf & \neginf &\posinf\\
\hline\Tstr
&\rv_1 &= \sqr{\uvar}+\sqr{\wvar} & u^2+v^2& \neginf & 0  & 0
\\[0.5ex]
&\rv_2 
&= (L+c\rv_{-3})^2 & (L+c\lam)^2 & 0  & \neginf & 2
\\
&\rv_3 &= \Diff{\lamvar}{2} &\lam'' & 
 2  & \neginf & 0 
\\
f_6= &\rv_4&=  \rv_1-\rv_2+\rv_3&u^2+v^2-(L+c\lam)^2+\lam'' & 0 &0 & 0\\
\hline\\[-6ex]
\end{array}
\end{align*}
\end{table}
\end{example}
For convenience in the notation below, by  
``$v$ is an algebraic function $\phi$ of a set $U$ of variables $u$'' we mean $v = v_r$ in \rf{code} is obtained through a $\phi = \phi_r$ with arguments $u\in U = \set{ v_{i_1}, \ldots, v_{i_m}}$.

\smallskip

We propagate variable offsets through the code list of a DAE based on the following lemma.
\begin{lemma}\label{le:propagateoffset}
\begin{enumerate}[(i)]
\item If  $v$ is $\rv_{-n} = t$, then $\varoffs{i}{\rv_{-n}} =
\varoffs{i}{t} = 
\posinf$.
\item 
If $v$ is $\rv_{j-n} = x_j$ then
\begin{align*}
\varoffs{i}{\rv_{j-n}} = \varoffs{i}{x_j} = \begin{cases}
\sij{i}{j} & \text{ if } j\in \Mset_i \\
\posinf & \text{ otherwise.}
\end{cases}
\end{align*}
\item 
If $\rv$ is an algebraic function $\phi$ of a set $U$ of  variables $\uv$, then
\begin{align*}
\varoffs{i}{\rv} & = \min_{u\in U} \varoffs{i}{\uv}.
\end{align*}

\item  
If $\rv = d^pu/dt^p$, where $p\ge0$, then
\begin{align*}
\varoffs{i}{\rv}&= \varoffs{i}{\uv}-p. 
\end{align*}

\item If $\rv$ is an output variable (corresponding to  $f_i$), then $\varoffs{i}{v} =  \varoffs{i}{f_i}= 0$.
\label{en:rul5}
\end{enumerate}
\end{lemma}

\smallskip 
\begin{proof}
\begin{enumerate}[{(i)}]
\item $v_{-n}=t$ does not depend on any $x_j$, so 
 $\sigma(\rv_{-n})$ contains only $\neginf$'s, and $\varoffs{i}{\rv_{-n}} = \infty$ for all $\rng{i}{1}{n}$.
\item 
 $\sigma_k(x_j) = 0$ if $k=j$ and $\neginf$ otherwise. Hence
\begin{align*}
\varoffs{i}{\rv_{j-n}}= \varoffs{i}{x_{j}} = \min_{k\in \Mset_i} \bigl(\sij{i}{k} - \sigma_k(\rv)\bigr) 
= \begin{cases}
\sij{i}{j} - \sigma_j(x_j) = \sij{i}{j} & \text{if } {j\in \Mset_i}\\
\posinf &\text{otherwise.}
\end{cases}
\end{align*}

\item Using\footnote{proved in \cite{nedialkov2007solving}} \  $\sigma_j(\rv) = \max_{u\in U}\sigma_j(\uv)$, 
\begin{align*}
\varoffs{i}{\rv} &= \min_{j\in \Mset_i}\bigl(\sij{i}{j} - \sigma_j(\rv)\bigr) 
= - \max_{j\in \Mset_i}\bigl(\sigma_j(\rv)-\sij{i}{j} \bigr) 
 = - \max_{j\in \Mset_i}\bigl( \max_{u\in U}\sigma_j(\uv)-\sij{i}{j} \bigr) 
\\
& = - \max_{j\in \Mset_i} \max_{u\in U}\bigl(\sigma_j(\uv)-\sij{i}{j} \bigr) 
 = -  \max_{u\in U}\max_{j\in \Mset_i}\bigl(\sigma_j(\uv)-\sij{i}{j} \bigr)  
\\
& = \min_{u\in U} \min_{j\in \Mset_i}\bigl(\sij{i}{j} -\sigma_j(\uv)\bigr)
 = \min_{u\in U} \varoffs{i}{\uv}.
\end{align*}

\item 
$\varoffs{i}{\rv} = \min_{j\in \Mset_i} \bigl(\sij{i}{j}-(\sigma_j(\uv)+p)\bigr)
= \varoffs{i}{\uv}-p.
$

\item 
If $\rv = f_i$, then
$\varoffs{i}{\rv} = \min_{j\in \Mset_i} \bigl((\sij{i}{j} - \sigma_j(f_i)\bigr) = \sij{i}{j} - \sij{i}{j} = 0.
$
\end{enumerate}
\qquad~\end{proof}

\subsection{QLity of a variable}\label{ss:vartype}
\text{}

\begin{definition}\label{def:typevec}
The {\em QLity} of  a variable $\rv$ in a code list of an $f_i$ is 
\begin{equation*}
\vartype_i(\rv)= 
\begin{cases} 
\typeI & \text{if $\rv$ does not depend on any $y\in \yiset$};\\
\typeN & \text{if $\rv$ depends nonlinearly on some 
$y\in \yiset$}; \eqntxt{and}\\
\typeL & \text{if $\rv$ depends only linearly on 
$y\in \yiset$}.
\end{cases}
\end{equation*}
\end{definition}
 
 \begin{remark}\rm
As in Definition~\ref{def:sigvec},  we mean ``formal" dependence. For example
$v$ depends nonlinearly on $x''$ in 
$v = (x'')^2 + x'' +x\lam -(x'')^2$, while the true dependence is linear. 
 \end{remark}

If $v$ is an algebraic function of a set $U$ of variables $u$, denote 
\begin{align}
U_0 = \set{u\in U \mid \varoffs{i}{u} = 0}.\label{eq:U0}
\end{align}

We propagate QLity  information based on the following  lemma.

\vspace{2pt}

\begin{lemma}
\begin{enumerate}[(a)]
\item 
  $\varoffs{i}{\rv}>0$ iff $\ttype_i(\rv) = \typeI$.
  
\item If $ \varoffs{i}{\rv}=0$, we have the following cases.
\begin{enumerate}[(i)]

\item If $\rv$ is an input variable $x_j$, then $\sij{i}{j}=0$ and
\begin{equation*}
\vartype_i(x_j)=
\typeL .
\end{equation*}

\item 
If $\rv$ is an algebraic function $\phi$ of a set $U$  of variables $\uv$  then, see \rf{U0}, 
\begin{align*}
\vartype_i(\rv)= 
\begin{cases}
\typeL & \text{if  $\ttype_i(\uv)=\typeL$ for all $u\in U_0$ and   $\phi$ is linear in all $u\in U_0$};    \ and \\ 
\typeN & \text{otherwise}.
\end{cases}
\end{align*}
\item If $\rv = d^pu/dt^p$, where $p>0$, then 
\begin{align*}
\vartype_i(\rv)= 
\typeL . 
\end{align*}
\item  If $\rv$ is an output variable (corresponding to  $f_i$), then $\ttype_i(v) = \typeL$ or \typeN.\\[-0.5ex]

\end{enumerate}
\end{enumerate}
\end{lemma}

\begin{proof}
\begin{enumerate}[(a)]
\item 
We have $\varoffs{i}{\rv} > 0$ iff   $\sij{i}{j} > \bcode{\rv}{j}$ for all  $j\in \kset{i}$ iff $\rv$ does not depend on 
any $y \in \yiset$ iff $\ttype_i(\rv) = \typeI$.

\item 
If $\varoffs{i}{\rv} = 0$, then $\rv$ depends on at least one $y\in \yiset$. 
\begin{enumerate}[{(i)}]
\item  If $v=x_j$ then $\sij{i}{j} = \sigma_{j}(x_j) = 0$.
That is, the highest-order derivative of $x_j$ in  $f_i$ is $\sij{i}{j} = 0$, and $\ttype_i(x_j)=\typeL$.
%

\item

If $\ttype_i(u)=\typeL$ for all $u\in U_0$, and  $\phi$ is linear in those $u$'s, then 
$v$ depends only linearly on $y\in \yiset$, and $\ttype_i(v)=\typeL$. 
Otherwise, $\ttype_i(u) = \typeN$ for some $u\in U_0$
or $\phi$ is nonlinear 
in some $u\in U_0$. In both cases, $\ttype_i(v) = \typeN$.

  \item Since $\varoffs{i}{\rv}=0$, $\rv$ will not be differentiated further in the code list (of $f_i$), and since $p>0$, $\rv$ depends linearly on its highest-order derivatives. 
Hence $\ttype_i(\rv) = \typeL$. 

\item This follows from (a) and definition~\ref{def:typevec}.
\end{enumerate}
\end{enumerate}
\qquad~\end{proof}

 \pagebreak
 Since $\varoffs{i}{v}=0$ implies $v$ will not be differentiated   further, we have 
 \begin{corollary}
 If $\vartype_i(\rv) = \typeN$ for some intermediate $\rv$ in a code list for $f_i$,  then 
$\vartype_i(f_i) = \typeN$.
 \end{corollary}

That is, we can conclude from the first $\rv$ with $\ttype_i(\rv)= \typeN$ that $f_i$ is NQL.

\begin{example}
\rm
Consider $f_1 = A$ in \eqref{eq:mod2p}. Here $\kset{1}=\setbg{1,3}$, $\yset{1}=\setbg{x'',\lam}$; we wish
to determine if $A$ is 
QL in $x''$ and $\lambda$. In Table~\ref{tbl:A},
$\vartype_1(f_1)=\typeL$, and $A$ is QL.

 \begin{table}[ht]\footnotesize
 \caption{\label{tbl:A}Code list and propagation of offsets and QLities for $f_1=A$}
 \vspace{-10pt}
 \setcounter{vvc}{1}
 \begin{align*}
\begin{array}{lll@{\hskip 12pt}lcc}
\multicolumn{3}{c}{\text{code list}} &\multicolumn{1}{c}{\text{expression}}&
	 \varoffs{1}{\rv_r} 
	& \ttype_1(\rv_r)
	 \\
	  \hline \Tstr	  
&\xvar &= x & x&2 & \typeI \\
&\lamvar &= \lam &\lam& 0 & \typeL\\	  
\hline
\Tstr
&\vv &= \Diff{\xvar}{2} &x''  & 0 & \typeL\\
&\rv_2 &= \xvar \lamvar &x\lam&   0 & \typeL\\
f_1 = &\rv_3 \ &= \rv_1+\rv_2& x''+x\lam &0 & \typeL \\
\hline
 \end{array}
\end{align*}

\end{table}
\end{example}

\begin{example}\label{ex:evaluateF}\rm
Now consider $F$.
  We have $\kset{6} = \setbg{3, 4}$, $\yset{6} = \setbg{u, \lam''}$, and we need to determine if $F$ is QL in $\lam''$ and $\uv$.
In Table~\ref{tbl:F}  $\vartype_6(\rv_1)=\typeN$,  and we can conclude from it that  $F$ is NQL.

\begin{table}[ht]\footnotesize
\setcounter{vvc}{-2}
\caption{\label{tbl:F}Code list and propagation of offsets and QLities for $f_6 = F$}
\vspace{-10pt}
\begin{align*}
\begin{array}{cll@{\hspace{20pt}}l@{\hspace{20pt}}ccc}
\multicolumn{3}{c}{\protect\text{code list}}    & 
\text{expression}   & 
\varoffs{6}{\rv_r} &
\ttype_6(\rv_r)  \\ \hline
	 \Tstr
&\lamvar & = \lam & \lam&2 & \typeI\\
&\uvar & = u &u& 0 & \typeL\\
&\wvar & = v &v& \posinf & \typeI\\
\hline
\Tstr
& \rv_1 &= \sqr{\uvar}+\sqr{\wvar} & u^2+v^2 & 0 & \typeN
\\
&\rv_2 & = (L+c\lamvar)^2 & (L+c\lam)^2 & 2& \typeI
\\
&\rv_3 & = \lam''  & \lam'' & 0&\typeL
\\
f_6 =&\rv_{4} &= \rv_1-\rv_2+\rv_3 &u^2+v^2- (L+c\lam)^2+\lam''&0&\typeN
\\
\hline
\end{array}
\end{align*}

\end{table}

\end{example}

\newcommand{\qlalgo}{{\sc ql\_analysis}\xspace}
\subsection{Quasilinearity analysis algorithm}\label{ss:algor}

From \S\ref{ss:varoffs} and \S\ref{ss:vartype}, we derive 
 in Figure~\ref{fig:algor2} algorithm \qlalgo,  which determines if an $f_i$ is QL/NQL. 
 Since $\varoffs{i}{\rv}>0$ iff  $\vartype_i(v)=\typeI$,
in practice, we need to propagate only $\typeL$ and $\typeN$.

\begin{figure}[ht]
\begin{algo}{\qlalgo}
\noindent\label{alg:algor}
 \Input
\\
\>code list for evaluating an $f_i$ in \rf{maineq}, $\sij{i}{j}$ for all $\rng{j}{1}{n}$
\\
\Output\\
\>$\vartype_i(f_i)$\\
\Compute\\
\>$\Mset_i \assa \set{ j \mid \sij{i}{j} = d_j-c_i}$\\
\>\FOR each variable $\rv$ in the code list
\\
\>\>\IF $\rv$ is an input variable $x_j$  \THEN
\\
\>\>\>\IF $j\in \Mset_i$ \THEN \\
\>\>\>\>$\varoffs{i}{\rv}\assa \sij{i}{j}$\\
\>\>\>\>\IF $\varoffs{i}{\rv}=0$ \THEN $\vartype_i(\rv) \assa \typeL$ \\
\>\>\>\ELSE $\varoffs{i}{\rv} \assa \posinf$\\
\>\>\ELSEIF $\rv$ is  $t$ \THEN \\
\>\>\> $\varoffs{i}{\rv}\assa \posinf$ \\
\>\>\ELSEIF $\rv$ is an algebraic function $\phi$ of a set $U$ of previous variables $\uv$ \THEN
\\
\>\>\>$\varoffs{i}{\rv}  \assa  \min_{u\in U} \varoffs{i}{\uv} $\\
\>\>\>\IF $\varoffs{i}{\rv} =0$\\
\>\>\>\>$U_0 \assa \set{ u\in U\mid \varoffs{i}{u} = 0}$\\
\>\>\>\>\IF $\ttype_i(u)=\typeL$ for all $u\in U_0$ and $\phi$ is linear in $U_0$ \THEN \\
\>\>\>\>\>$\vartype_i(\rv) \assa \typeL$\\
\>\>\>\>\ELSE $\vartype_i(f_i) \assa \typeN$, return\\
\>\>\ELSE \qquad \verb+%+ $\rv$ is $d^p u/dt^p$ 
\\
\>\>\>$\varoffs{i}{\rv} \assa \varoffs{i}{\rv}-p $\\
\>\>\>\IF $\varoffs{i}{\rv}=0$ \THEN $\vartype_i(\rv) \assa \typeL$ 
 \end{algo}
\caption{\label{fig:algor2} Algorithm for determining if an $f_i$ in the basic scheme is QL.}
\end{figure}

\subsection{Quasilinearity of a fine block}
\label{ss:qlafine}

At local stage $k_l >0$, system \rf{solb2} is linear in 
$\xjkb{k}{l}$. We are interested in determining if
it is linear in $\xjkb{k}{l}$ at local stage $k_l\le 0$.
Similarly to the definition of $\yiset$ in  \rf{yiset}, 
used to define  quasilinearity of an $f_i$ in \rf{solb2}, let 
\begin{align*}
\ziset &= \setbg{ x_j^{(\sij{i}{j})} \mid \blkl{j}{l} \text{ and } \sij{i}{j} =d_j-c_i}.
\end{align*}
 
\begin{definition}\label{def:fiql2}
Equation $f_i=0$ with $\blkl{i}{l}$ is QL, if it is linear in all $y\in\ziset$, and NQL otherwise.
\end{definition} 

\begin{definition}\label{def:daeql2}
System \rf{solb2} is NQL if it contains an undifferentiated NQL $f_i$, and QL otherwise.
\end{definition}

When considering block $l$,  algorithm \qlalgo can be readily adapted by changing the line 
\[
\Mset_i \assa \set{ j \mid \sij{i}{j} = d_j-c_i}
\eqntxt{to}
\Mset_i \assa \setbg{j\mid \blkl{j}{l} \text{ and }
\sij{i}{j}=d_j-c_i}
\]
and keeping the rest of this algorithm the same.
 
\begin{example} \rm Consider block $l=3$. 
For equation $f_3 = F=0$, $\blkl{3}{3}$.
Then 
\begin{align*}
\Mset_3 &= 
  \setbg{j\mid \blkl{j}{3} \text{ and } \sijc{3}{j}=d_j-c_3 } = 
\set{ 3 },
\end{align*}
 and we need to determine if $F$ is QL in $x_3=u$.
Now the index of $x_6=\lam$ is $6\notin \Mset_3$ and therefore $\lam''\notin Z_3$; cf. Example~\ref{ex:evaluateF}. We show in Table~\ref{tbl:Fb} the initialization and propagation of offsets and QLity  information for the code list in Table~\ref{tbl:F}.

\begin{table}[ht]\footnotesize
\caption{\label{tbl:Fb}Code list and propagation of offsets and QLities for $f_3 = F$ of block 3 in the BTF of \protect\rf{mod2p}}
\vspace{-10pt}
\begin{align*}
\begin{array}{cll@{\hspace{20pt}}l@{\hspace{20pt}}ccc}
\multicolumn{3}{c}{\protect\text{code list}}    & 
\text{expression}   & 
\varoffs{3}{\rv_r} &
\ttype_3(\rv_r)  \\ \hline
	 \Tstr
&\rv_{-5} & = v & v&\infty & \typeI\\
&\rv_{-3} & = u &u& 0 & \typeL\\
&\rv_0 & = \lam&\lam& \posinf & \typeI\\
\hline
\Tstr
& \rv_1 &= \sqr{\rv_{-3}} + \sqr{\rv_{-1}}& u^2 +v^2& 0 & \typeN\\
&\rv_2 & = (L+c\lamvar)^2 & (L+c\lam)^2 & 2& \typeI
\\
&\rv_3 & = \lam''  & \lam'' & 0&\typeL
\\
f_3 =&\rv_{4} &= \rv_1-\rv_2+\rv_3 &u^2+v^2- (L+c\lam)^2+\lam''&0&\typeN
\\
\hline
\end{array}
\end{align*}

\end{table}

\end{example}

\section{Overall algorithm}\label{sc:overall}
In  \S\ref{ss:init} we present first an algorithm for finding the indices of derivatives that need to be initialized,
and then in \S\ref{ss:bsa} we    give an overall algorithm implementing the block solution scheme.


\subsection{Initialization}\label{ss:init}

Consider block $l$. For $k$ such that $k_l=k+\Kl<0$, we need to initialize 
$\xjkb{k}{l}$, that is, give values for 
\begin{align}
\Setbg{  {x_j^{(r)}} \mid 
\blkl{j}{l} \text{ and } r = k+d_j  \ge 0}.\label{eq:initval}
\end{align}
We have $k+d_j = \hk+\hd_j \ge 0$ when 
\[
k \ge -\max_{\blkl{j}{l}}d_j \eqntxt{or equivalently}
\hk \ge -\max_{\blkl{j}{l}}\hd_j.
\]
When $k_l < -\max_{\blkl{i}{l}} \hc_i$, there are no equations at this stage and in this block, and \rf{initval} are {\em initial values}, otherwise \rf{initval} are {\em initial guesses} (trial values).

When $k_l=0$ and \rf{solb2} is NQL (see \S\ref{ss:qlafine}), also initial guesses for 
\[
\Setbg{ {x_j^{(\hd_j)}}  \mid 
\blkl{j}{l} }
\]
are required. 

 \newcommand{\inalgo}{{\sc fine\_block\_initialization}}

This is expressed by  algorithm \inalgo, Figure~\ref{fig:algor}, which  finds the set of indices 
$(j,r)$  for derivatives $x_j^{(r)}$ that require initial values and the set of indices for derivatives that require initial guesses. By \cite[Theorem 4.5]{NedialkovPryce2012a},
 these are   minimal sets.

\begin{figure}[ht]
\begin{algo}{\inalgo}\label{alg:init}\Input\\
\>local  canonical offsets $c_i^*$, $d_j^*$,   \quad$\rng{i,j}{1}{N_l}$\\
\>$\qlvec{} =1$ if block   is QL and 0 otherwise\\
\>$b = \sum_{i=1}^{l-1}N_i$, block $l$ starts at index $b+1$
\\
\Output
\\
\>$\ivsindx$ set of indices for derivatives that require
{\em initial values}  \\
\>$\igsindx$ set of indices for derivatives that require
{\em initial guesses}  \\
\Compute
\\
\>$\ivsindx \assa \emptyset$\\
\>$\igsindx \assa \emptyset$\\ 
\>\FOR $q \assa-\max d_j^*$ \  to\  \ $-\qlvec{} $\\[0.5ex]
\>\>$M \assa \set{j\mid q+d_j^* \ge 0}$\\
\>\> $J \assa \setbg{ (j+b,q+d_j^*) \mid  j\in M}$
\\
\>\>\IF $q < -\max c_i^*$ \ \THEN \qquad  (no equations)
\\
\>\>\>\,$\ivsindx \assa \ivsindx\cup J$\\
\>\>\ELSE $\igsindx \assa \igsindx\cup J  $
\end{algo}
\caption{\label{fig:algor}Initialization for a fine block. 
To simplify the notation in this algorithm, for block $l$ we denote by $c_i^*$ and $d_j^*$ its local canonical offsets and assume they are indexed from $1$ to $N_l$;
that is $c_i^* = \widehat c_{b+i}$ and $d_j^* = \widehat d_{b+j}$. 
}
\end{figure}

\begin{example}\rm For the blocks in our example,  this algorithm produces sets as follows.
\begin{enumerate}[$-$]

\item 

Block $l=4$ is QL,  
$\gamma  = 1$, $b = 3$,  $\max d_j^* = \max c_i^* = 2$.
When $q = -2$, $M= \set{1,2}$, 
$\igsindx = J =  \set{ (4,0), (5,0)}$; when $q = -1$, 
$M$ is the same, $J =  \set{ (4,1), (5,1)}$, 
and   $$\igsindx = \set{ (4,0), (5,0), (4,1), (5,1)} 
\eqntxt{and} \ivsindx = \emptyset.$$

\item Block $l=3$ is NQL, $\gamma = 0$, $b = 2$, $\max d_j^* = 0$,
and $\max c_i^* = 0$. We have one iteration with $q=0$, $M = \set{1}$, and $$\igsindx = \set{ (3,0)}
\eqntxt{and} \ivsindx = \emptyset.$$

\item Block $l=2$ is QL, $\gamma = 1$, $\max d_j^* = 0$, 
the loop does not execute, and $\igsindx = \emptyset$ and $\ivsindx = \emptyset$

\item
Block $l=1$ is NQL,  
$\gamma  = 0$, $b = 0$,  $\max d_j^* = 3$, and $\max c_i^* = 0$.
After iterations $q = -3, -2, -1$, we have 
$$
\ivsindx = \set{ (1,0), (1,1), (1,2)},
$$
and after iteration $q = 0$,
\[\igsindx = \set{ (1,3)};
\]
$M = \set{1}$ through all iterations.
 
\end{enumerate}
Since the variables are in the order $v,u,\mu, x,y,\lam$,
the above $\igsindx$ and $\ivsindx$ sets 
imply that 
\begin{align*}
\begin{array}{ll}
x,y, \, x', y'\, u,\, v''' &\text{ require initial guesses and};\\
v, v', v''                   &\text{ require initial values}.
\end{array}
\end{align*}
\cf Table~\ref{tbl:solschbtf}.
\end{example}

\newcommand{\solvealgo}{{\sc solving\_for\_derivatives}}

\subsection{Block scheme algorithm}\label{ss:bsa}
Here we assume  that $\Sigma$, fine BTF, canonical local and global offsets are computed, the quasilinearity 
of  each $f_i$ is determined, and appropriate derivatives are given initial values or guesses (according 
algorithm \qlalgo).
Also, we  assume that the $\Kl$ are available (see the companion paper).
Finally, algorithm \solvealgo, Figure~\ref{alg:final}, gives the algorithm for 
stage $k$ of the block scheme.

\begin{figure}[ht]
\begin{algo}{\solvealgo}
\Input\\
\>$k$ stage number\\
\>$K_l$ for each block $l$\\
\>$\hc$, $\hd$ local canonical offsets\\
\>$\gamma_i=1$ if $f_i$ is QL, 0 otherwise, $\rng{i}{1}{n}$\\
\>$B_l$ for each block $l$
\\
\Output\\
\>$\~x_{J_k} = (\xjkb{k}{1}, \ldots, \xjkb{k}{p})$\\
\Compute\\
\>\FOR $l = p, p-1,\ldots, 1$\\
\>\>$k_l \assa  k + K_l$\\
\>\>\IF $k_l< -\max_{\blkl{i}{l}} \hc_i$ \THEN \  \CONTINUE\\
\>\>$\ikb{k}{l} \assa \setbg{ (i,r)  \mid  \blkl{i}{l} \text{ and }r =  k_l+\hc_i\ge0}$
\\
\>\>$\jkb{k}{l} \assa \setbg{ (j,r) \mid  \blkl{j}{l}
\text{ and } r=\hk+\hd_j\ge0}$\\
\>\>\IF there is $(i,0)\in \ikb{k}{l}$ and $\gamma_i=0$ \THEN\\
\>\>\>qlflag $\assa$ \FALSE\\
\>\>\ELSE qlflag $\assa$ \TRUE\\
\>\>\IF $\hk<0$ \THEN \\
\>\>\>\IF qlflag $=$ \FALSE\\
\>\>\>\>\,solve $\min \| \xjkbig{k}{l} - \xjkb{k}{l}\|_2\text{ \ s.t. \ } \text{nonlinear }\Fikb{k}{l}(\xjkb{k}{l}) = 0$\\
\>\>\>\ELSE solve $\min \| \xjkbig{k}{l} - \xjkb{k}{l}\|_2\text{ \ s.t. \ } \text{linear \phantom{non}}\Fikb{k}{l}(\xjkb{k}{l}) = 0$\\
\>\>\ELSEIF $\hk=0$ and qlflag $=$ \FALSE \THEN \\
\>\>\>\,solve nonlinear $\Fikb{k}{l}(\xjkb{k}{l})=0$\\
\>\>\ELSE solve linear\phantom{non} $\Fikb{k}{l}(\xjkb{k}{l})=0$
\end{algo}
\caption{Algorithm for computing derivatives at stage $k$ using fine BTF of the DAE.\label{alg:final} 
When $k_l<0$, $\Fikb{k}{l}(\xjkb{k}{l}) = 0$ is underdetermined and determined otherwise.}
\end{figure}

\newcommand{\bigo}{{O}}
\newcommand{\nopers}{\tau}

\section{Conclusion}\label{sc:concl}
We showed how to perform QL analysis of a DAE and how to combine it with its fine block triangularization to obtain a solution scheme. Provided the problem can be decomposed into many small fine blocks, this scheme  results in solving much smaller problems compared to the basic scheme. For instance, for the distillation column in  \cite{NedialkovPryce2012b}, we have 52 blocks of size 1 and 11 blocks of size 7. With the basic scheme, we would solve systems of size 129 at stages $\ge 0$, while with the block scheme, the largest system is of size 7. Furthermore, a DAE that is NQL may decompose  into subproblems  that are all QL: this is the
case e.g. with the Chemical Akzo Nobel problem,  a NQL DAE of size 6 that decomposes into 6 linear equations of size 1;
see \cite{NedialkovPryce2012b,NedialkovPryce2012a}.

In the context of computing a numerical solution using Taylor series, the block scheme has the advantage that we could exploit multiple processing units (CPUs, cores) to pipeline the computation of Taylor coefficients (TCs).  
Profiling the \daets solver \cite{nedialkov2008solving} has shown that typically more than 80\%, and in some cases above 90\%, of the total computing time is spent in  automatic differentiation (AD) for evaluating  TCs for the $f_i$'s.  Our hope is that, by assigning blocks of the DAE to different processing units and computing in a block-wise manner, we can reduce the overall time in AD to compute these coefficients.
We outline this pipelining idea here and discuss briefly some of the challenges involved.

 Using AD, to evaluate the $(k+c_i)$ TC of an $f_i$ requires $\bigo(k+c_i)$ operations, where the constant in the big-$\bigo$ notation depends on   
the number of nonlinear operations 
($\times$, $\div$, $\sin$, $\exp$, etc.) in $f_i$ \cite{griewank2008evaluating,Moor66a}.
(The work is $O(1)$ if $f_i$ contains only $+, -$, and  $d^p/dt^p$.)

Denote this constant by $\tau_i$ and write $O(k+c_i) = \tau_i(k+c_i)$. 
Using the basic scheme, at stage $k\ge 0$  we need to evaluate  $N$ such coefficients. Thus,  the work is 
\begin{align*}
\begin{split}
\sum_{i=1}^N \nopers_i (k+c_i) &= 
k\sum_{i=1}^N \nopers_i
 + \sum_{\stackrel{i=0}{c_i\neq 0}}^N\nopers_ic_i = 
	\nopers k + \sum_{\stackrel{i=0}{c_i\neq 0}}^N\nopers_ic_i 
	=O(\nopers k),
	\end{split}\label{eq:costfi}
\end{align*}
where $\nopers = \sum_{i=1}^N\nopers_i$ depends on the problem, and if the $c_i$'s are canonical, 
$\sum_{\stackrel{i=0}{c_i\neq 0}}^N\nopers_ic_i$ depends on the problem as well. To evaluate $q+1$ coefficients for stages $\rng{k}{0}{q}$, the work in AD is $\sum_{k=0}^q O(\tau k) = O(\tau q^2)$.

For our example problem, we could assign the $3\times 3$ fine block to one core, say $\text{P}2$,  and the coarse block consisting of the 3 fine blocks to another core, say P$1$. 
Then when $\text{P}2$ is at stage $k$,  $\text{P}1$ would be at stage $k-1$. When $\text{P}2$ finishes its last stage $q$, $\text{P}1$ needs to complete  stage $q$. 
If we assume that the work per each of these cores/blocks is $O(\tau q^2/2)$, that is the constant is $\tau/2$, $\text{P}2$ would finish in $O(\tau q^2/2)$ and $\text{P}1$ would finish in $O(\tau q/2)$ after $\text{P}2$. Under these  simplistic assumptions, and ignoring rest of the work (e.g. in linear algebra), we would expect the computation time to be reduced to about half. 

It is desirable to assign blocks to processing units and perform a pipelined computation, such that each stage of the pipeline takes about the same amount of time and there are no ``gaps'' in the pipeline. However, developing such a method is a challenging task on its own.
%

One difficulty is estimating the work per TC of an $f_i$. Another is, if  the number of fine blocks is much larger than the number of processing units, how to agglomerate   blocks such that each unit does nearly the  same amount of work.  A third challenge comes from the structure of the 
problem. For example,   assume that fine blocks 
3, 2,  and 1 are assigned to cores $\text{P}3$, $\text{P}2$, and $\text{P}1$, respectively. %
\begin{table}[ht]
\caption{\label{tbl:pipe}Pipelining blocks 3,2, and 1. 
$k, w:z$ means solving at stage $k$ for $w$ which depends (at least) on $z$. }
\vspace{-12pt}
{\footnotesize
\[
\begin{array}{l|c|c|c|c|c|c|c|c}
\text{core}& \multicolumn{7}{c}{\text{time slots }1,2,\ldots}\\ \hline \Tstr
\text{P}3  & 0, u'': v''& \times & \times &   1, u''':v'''&\times&\times & 2, u^{(4)}:v^{(4)} & \cdots
\\
\text{P}2 & & 0, \mu: u'' & \times &\times &1, \mu': u''' &
\times& \times&\cdots
 \\
\text{P}1 & && 0, v''': \mu & \times  & \times &1, v^{(4)}: \mu' & \times & \cdots 
\end{array}
\]
}
\end{table}
In Table~\ref{tbl:pipe}, $\text{P3}$ needs $v^{(k)}$ to find $u^{(k)}$, 
and $\text{P}1$ needs $\mu^{(k)}$, which in turn depends on $u^{(k)}$ to find $v^{(k)}$.
That is, $\text{P}3$ has to wait for $v^{(k)}$ to be determined by $\text{P}1$, which leads to the above 
 gaps, marked by $\times$. 
 
\appendix
\section{QL analysis: implementation issues}\label{ss:oov}
Algorithm \qlalgo\ (Figure \ref{alg:algor}) can be implemented using operator overloading or source code translation as follows.  

In an initialization phase, we associate with 
each input variable $x_j$   an $n$-vector with $i$th component 
\[
\varoffs{i}{x_j} = \begin{cases}
\sij{i}{j}  & \text{if } \sij{i}{j} = d_j-c_i  \text{ and $i$ and $j$ in the same block}\\
\posinf & \text{otherwise; }
\end{cases} 
\]
see Table~\ref{tbl:full}.
According to Lemma~\ref{le:propagateoffset},
for a unary operator, the offset vector remains the same; for each binary operator, we take a component-wise minimum of the associated offset vectors; and for the 
$d^p/dt^p$ operator, we can subtract $p$ from each component.
Since 
\[
\ttype_i(x_j) = \begin{cases}
\typeL  & \text{if } \varoffs{i}{x_j} = 0 \text{ and}\\
\typeI   & \text{otherwise},
\end{cases}
\]
we can assume that offset 0 encodes $\typeL$ and offset $>0$ encodes $\typeI$. Further, we can use e.g. 
``offset'' $-1$ to encode $\typeN$.

For an intermediate variable $v$, if $\ttype_i(v)=\typeN$, then the subsequent intermediate variables in the code list for the $f_i$ will be of QLity $\typeN$. Hence, for a $v$ with $\ttype_i(v) = \typeN$, we can set the $i$th component of its offset vector to  $-1$. 
Then, if the $i$th component of the offset vector of $f_i$ is 0, $\ttype_i(f_i) = \typeL$, and if it is $-1$, 
$\ttype_i(f_i) = \typeN$.

In passing, we note that, if a variable $v$ is in the code list of $f_i$ but not in the code list of $f_j$, $i\neq j$, then $\varoffs{j}{v}$ is ignored. However, we propagate the whole $n$-vector, as we do not know in advance (in particular with operator overloading) if $v$ appears in the code list of more than one $f_i$.
  
\newcommand{\oftI}[1]{#1}
\newcommand{\oftL}[1]{#1}
\newcommand{\oftN}[1]{#1}
\renewcommand{\bf}{}
\setcounter{vvc}{1}
\renewcommand{\vv}{\rv_{\thevvc}\addtocounter{vvc}{1}}
\renewcommand{\inftag}{-}

\renewcommand{\fbox}[1]{\colorbox{lightgray}{#1}}
\begin{table}[ht]\footnotesize
\caption{\label{tbl:full}Propagation of variable offsets and QLity information  through code list of \protect\rf{mod2p}. The QLity of each equation (output variable) is marked  by 
$\fbox{\protect\phantom{u}}$\,; $-$ denotes $\infty$.}
\vspace{-5pt}
\begin{align*}
\begin{array}{crl|l|rrr rrr}
         &  \multicolumn{2}{c|}{\text{code list}}&\multicolumn{1}{c|}{\text{expression}}& \multicolumn{1}{c}A & \multicolumn{1}{c}B &\multicolumn{1}{c}{C} & \multicolumn{1}{c}D & \multicolumn{1}{c}E & \multicolumn{1}{c}F 
         \\
                  &   && & 
                  \multicolumn{1}{c}{\alpha_1} & 	 \multicolumn{1}{c}{\alpha_2} &
                 \multicolumn{1}{c}{\alpha_3} & 
                 \multicolumn{1}{c}{\alpha_4} & 
                 \multicolumn{1}{c}{\alpha_5} & 
                 \multicolumn{1}{c}{\alpha_6} 
             \\ \hline\Tstr
&  \xvar& = x& x& \oftI 2  &  \oftI\inftag &  \oftL 0&  \inftag & \inftag & \inftag  
\\
&   \yvar&=y &y& \inftag& 2 &   \tagL&  \inftag & \inftag & \inftag  
\\
&  \lamvar& = \lam &\lam& \tagL& \tagL & \inftag&  \inftag & \inftag & \inftag  
\\
& \uvar &= u &u&  \inftag&  \inftag & \inftag&  \tagIM{2} & \inftag& \tagL     
\\	
& \wvar& =v & v& \inftag&  \inftag & \inftag&  \inftag&  3 &\inftag   
\\	
& \muvar& =\mu &\mu&  \inftag&  \inftag & \inftag&  \tagL & \inftag & \inftag  \\
	\hline
	\Tstr
& \vv    &= \xvar'' &x'' & \bf \tagL & \inftag& -2&  \inftag & \inftag & \inftag  \\
& \vv  &= \xvar\lamvar&x\lam & 0& \tagL & \phantom{-} 0&  \inftag & \inftag & \inftag  \\
A& =\vv &= \rv_1+\rv_2 & x''+x\lam& \bf \fbox{0}& \tagL &  -2&  \inftag & \inftag & \inftag  
\\
\hline
\Tstr
&\vv &= \yvar'' & y'' & \inftag&\bf \tagL&  {-2}&  \inftag & \inftag & \inftag  \\
&\vv &= \yvar\lamvar & y\lam & 0&\bf 0 & 0&  \inftag & \inftag & \inftag  \\
& \vv&=\rv_4+\rv_5& y''+y\lam & \tagL& \bf\tagL &  -2&  \inftag & \inftag & \inftag  \\
& \vv&=\xvar'&x'& 1& \bf \tagIM{1} & {-1}&  \inftag & \inftag & \inftag  \\
&\vv&=\rv_7^2& (x')^2&1&\bf  \tagIM{1} &  {-1}&  \inftag & \inftag & \inftag  \\
& \vv&=\rv_6+\rv_8& y''+y\lam+(x')^2&\tagL&\bf \tagL & -2&  \inftag & \inftag & \inftag  
\\
B& = \vv&=\rv_9 - G&y''+y\lam+(x')^2-G& \tagL&\bf \fbox{0} & -2&  \inftag & \inftag & \inftag  
\\
\hline
\Tstr
& \vv &= \xvar^2 &x^2&2 & \tagIM{1} &  \bf\tagNL&  \inftag & \inftag & \inftag 
\\ 
& \vv &= \yvar^2 &y^2&\inftag & 2 & \bf \tagNL&  \inftag & \inftag & \inftag\\
& \vv &= \rv_{11}+\rv_{12}&x^2+y^2 &2 & 2 & \bf \tagNL&  \inftag & \inftag & \inftag\\
         \Tstr
C &= \vv&=\rv_{13}-L^2&x^2+y^2-L^2& 2& 2 & \bf \fbox{$-1$}&  \inftag & \inftag & \inftag \Tstr \\
\hline 
\Tstr
&  \vv &= \uvar'' &u''&  \inftag &  \inftag & \inftag&\inftag &  \inftag & -2 \\
&  \vv &= \uvar\muvar & u\mu & \inftag&  \inftag & \inftag&\bf  \tagL&  \inftag  & \tagL 
\\
D &= \vv &= \rv_{15}+\rv_{16}&u''+u\mu & \inftag&  \inftag & \inftag& \bf \fbox{0}& \inftag  & -2
\\
\hline
\Tstr
&  \vv & = \wvar''' & v'''& \inftag&  \inftag & \inftag&  \inftag  & \bf \tagL& \inftag
\\
& \vv & = (\wv''')^2 & (v''')^2 & \inftag&  \inftag & \inftag&  \inftag &\bf \tagNL& \inftag  \\
& \vv & = \wvar\muvar & v\mu & \inftag&  \inftag & \inftag&  0& 3 & \inftag\\
& \vv& = \rv_{19}+\rv_{20} & (v''')^2+v\mu&\inftag&  \inftag & \inftag&  \tagL& \bf \tagNL  &\inftag\\
E & = \vv & = \rv_{21}-G &(v''')^2+v\mu-G& \inftag&  \inftag & \inftag&  \tagL&\bf \fbox{$-1$}  &\inftag\\
\hline
\Tstr
& \vv & = \uvar^2 & u^2&\inftag&  \inftag & \inftag&  \inftag& \inftag & \bf \tagNL\\
& \vv & = \wvar^2 & v^2&\inftag&  \inftag & \inftag&  \inftag& 3 &\bf \inftag \\
& \vv & = \rv_{23}+\rv_{24} & u^2+v^2&
	 \inftag&  \inftag & \inftag&  \inftag& 3 & \bf \tagNL\\
& \vv & = c\lamvar & c\lam&
	 \tagL&  \tagL & \inftag& \inftag& \inftag& \inftag
	 \\
&\vv & = L+\rv_{26} &L+ c\lam&
	 \tagL&  \tagL & \inftag& \inftag& \inftag& \inftag
	 \\	
& \vv & = \rv_{27}^2 &(L+ c\lam)^2& 
	 \tagNL&  \tagNL & \inftag& \inftag& \inftag& \inftag
	 \\
& \vv & = \rv_{25} -\rv_{28}& u^2+v^2-(L+ c\lam)^2&
	\tagNL& \tagNL & \inftag&\inftag& 3&\bf \tagNL
	  \\	 	
& \vv & = \lamvar & \lam''&
	 -2& -2& \inftag& \inftag& \inftag& \inftag
	 \\
F &= \vv & = \rv_{29}+\rv_{30} & u^2+v^2-(L+ c\lam)^2+\lam''&
	-2&  -2 & \inftag&\inftag& 3&\bf \fbox{$-1$} 
\\ \hline
 \end{array}
\end{align*}

\end{table}

\bibliographystyle{siam}

\end{document}